%% file: main.tex
\documentclass[12pt]{amsart}
\usepackage{fullpage} 
\usepackage[T1]{fontenc}
\usepackage[utf8]{inputenc}
\input{preamble.tex}

\usepackage[ruled,vlined]{algorithm2e}
\usepackage{colortbl}
\usepackage{tabularx}

\title{Algorithms for numerical semigroups with fixed maximum primitive}
\author{Manuel Delgado, Neeraj Kumar}
\date{\today}

\address{CMUP--Centro de Matemática da Universidade do Porto, Departamento de Matemática, Faculdade de Ciências, Universidade do Porto, Rua do Campo Alegre s/n, 4169– 007 Porto, Portugal} 
\email{mdelgado@fc.up.pt, \commentgreen{neerajk095@gmail.com}} 
\thanks{The authors were partially supported by CMUP -- Centro de Matemática da Universidade do Porto, member of LASI, which is financed by national funds through FCT -- Fundação para a Ciência e a Tecnologia, I.P., under the project with reference UID/00144/2025, doi: \url{https://doi.org/10.54499/UID/00144/2025}. 
	The first author also acknowledges the Proyecto de Excelencia de la Junta de Andalucía (ProyExcel 00868).}
\begin{document}


\begin{abstract}
 	We present an algorithm to explore various properties of the numerical semigroups with a given maximum primitive. In particular, we count the number of such numerical semigroups and verify that there is no counterexample to Wilf's conjecture among the numerical semigroups with maximum primitive up to \(60\).
\end{abstract}

\maketitle

\section{Introduction}\label{sec:introduction}
\input{introduction}


\section{A few simple results}\label{sec:results}
\input{results}


\section{A naive algorithm}\label{sec:naive_algorithm}
\input{naive_algorithm}

\section{Optimization and implementation}\label{sec:optimized_algorithm}
\input{optimized_algorithm}

\section{Optimizations with counting purposes}\label{sec:optimizations_for_counting}
\input{optimizations_for_counting}

\section{Computational results}\label{sec:computations}
\input{computations}


\printbibliography
\end{document}

%% file: preamble.tex
\usepackage{amsfonts,amsmath,amsthm,amssymb}
\usepackage{enumitem} 
\usepackage{siunitx} 

\usepackage[hidelinks,linktoc=all]{hyperref}

 \usepackage[backend=biber,style=numeric,hyperref=true,isbn=false,doi=false]{biblatex}
\usepackage{pgf} 
\usepackage{graphicx}
\usepackage{subcaption}
\usepackage{adjustbox}

\usepackage{float} 
\usepackage{wrapfig} 

\newcommand{\commentgreen}[1]{{\color{green} #1}} 

\DeclareMathOperator{\Frobeniusoper}{f} 
\DeclareMathOperator{\multiplicityoper}{m} 
\DeclareMathOperator{\leftsoper}{L} 
\DeclareMathOperator{\primitivesoper}{P} 
\DeclareMathOperator{\genusoper}{g} 

\newcommand{\maxprimset}[1]{\mathcal{A}_{#1}}
\newcommand{\maxprimcard}[1]{{A}_{#1}}

\newcommand{\frobset}[1]{\mathcal{N}_{#1}}
\newcommand{\frobcard}[1]{{N}_{#1}}

\newtheorem{theorem}{Theorem}
\newtheorem{lemma}[theorem]{Lemma}

\newtheorem{proposition}[theorem]{Proposition}

\theoremstyle{remark}

\newtheorem*{remark*}{Remark}

\usepackage[many]{tcolorbox}    	
\definecolor{main}{HTML}{5989cf}    
\definecolor{sub}{HTML}{cde4ff}     

\tcbset{
    sharp corners,
    colback = white,
    before skip = 0.2cm,    
    after skip = 0.5cm      
}                           

\newtcolorbox{boxA}{
    enhanced, 
    boxrule = 0pt, 
    borderline = {0.75pt}{0pt}{main}, 
    borderline = {0.75pt}{1pt}{sub} 
}



%% file: introduction.tex
We denote the set of nonnegative integers by \(\mathbb{N}\). A \emph{numerical semigroup} \(S\) is a cofinite submonoid of \(\mathbb{N}\). The smallest finite set that generates \(S\) is referred to as the set of \emph{minimal generators} of \(S\), or the set of \emph{primitives} of \(S\) and we denote this set by \(\primitivesoper(S) \). The existence and uniqueness of \(\primitivesoper(S) \) is a well known result (see~\cite[Theorem 2.7]{Garcia-SanchezRosales1999PJM-Numerical}). The cardinality of \(\primitivesoper(S)\) is called the \emph{embedding dimension} of \(S\) and denoted by \(e(S)\). The smallest primitive of \(S\) is called the \emph{multiplicity} of \(S\) and is denoted by \(\multiplicityoper(S)\), and its largest primitive is called the \emph{maximum primitive} of \(S\).
A positive integer that does not belong to a numerical semigroup \(S\) is called a \emph{gap} of \(S\), and the total number of gaps of \(S\) is called the \emph{genus} of \(S\) and is denoted by \(\genusoper(S)\). The largest integer that does not belong to \(S\) is called the \emph{Frobenius number} of \(S\) and is denoted by \(\Frobeniusoper(S)\), and \(\Frobeniusoper(S)+1\) is referred to as the \emph{conductor} of \(S\) and is denoted by \(c(S)\). The elements of \(S\) smaller than \(\Frobeniusoper(S)\) are called the left elements of \(S\) and we denote the number of left elements of \(S\) by \(\ell(S) \). We refer to \( \lceil c(S)/\multiplicityoper(S)\rceil \) as the  \emph{depth} of \(S\), and \( \lceil \max P(S) /\multiplicityoper(S)\rceil \) as the \emph{primitive depth} of \(S\). We omit the mention of \(S\) when it is understood from the context and write \(\multiplicityoper,e, \genusoper , \Frobeniusoper, \ell, \) etc instead of \(\multiplicityoper(S), e(S), \genusoper(S), \Frobeniusoper(S), \ell(S),\) etc respectively. Wilf~\cite{Wilf1978AMM-circle} asked whether the inequality
 \[e\ell\geq c\]
holds for all numerical semigroups, and the affirmative answer to this is nowadays referred to as \emph{Wilf's conjecture}. 

Given a positive integer \(n\), we denote the set of numerical semigroups with Frobenius number equal to \(n\) by \(\frobset{n}\), and its cardinality by \(\frobcard{f}\). The \emph{counting} of numerical semigroups refers to the problem of enumerating the set of all numerical semigroups with a given property. For instance, the problem of computing \(\frobcard{n}\) is referred to as the problem of \emph{counting numerical semigroups by Frobenius number} or, in short, \emph{counting by Frobenius number}. We note that alternatively one can define the counting by Frobenius number as the estimation of the sequence \((\frobcard{n})\). Similarly, we define \emph{counting by genus} and \emph{counting by maximum primitive}.

Wilf~\cite{Wilf1978AMM-circle} posed the problem of estimating the value of \(\frobcard{n}\) as \(n\) becomes large. Backelin~\cite{Backelin1990MS-number} answers this by proving that \(\frobcard{n}\) grows asymptotically as a constant multiple of \(2^{n/2}\) where the constant depends upon the parity of \(n\). He further gives upper and lower bounds for \(\frobcard{n}\). More recently, Singhal~\cite{Singhal2022SF-Distribution} shows that \((\frobcard{2n})\) and \((\frobcard{2n+1})\) are monotonous sequences, and discusses other asymptotic properties of the set \(\frobset{n}\). Li~\cite{Li2023CT-Counting} considers the asymptotic growth of certain subsets of \(\frobset{n}\).

We denote the number of numerical semigroups with genus \(g\) by \(n_g\). The counting by genus gained interest when Bras-Amorós~\cite{Bras-Amoros2008SF-Fibonacci} conjectured a Fibbonacci-like behavior in the sequence \((n_g)\), in a paper where she computed \(n_g\) for  \(g\le 50\) and verified that there is no counter-example to Wilf's conjecture among the numerical semigroups with genus up to \(50\). A breakthrough occurred in a paper by Fromentin and Hivert~\cite{FromentinHivert2016MC-Exploring} where they computed \(n_g\) for \( g\le 67\) and verified that there is no counter-example to Wilf's conjecture among the numerical semigroups with genus up to \(60\). Furthermore, the value of \(n_g\) up to genus~\(76\) have recently been obtained by Bras-Amorós~\cite{Bras-Amoros2025pp-Exploring}. Moreover the verification of Wilf's conjecture up to genus \(100\) has been obtained recently by Delgado, Eliahou and Fromentin~\cite{DelgadoEliahouFromentin2025JoA-verification}. A pertinent detail from this verification is that unlike the previous verifications where all the numerical semigroups up to the given genus were explored, the methods used in~\cite{DelgadoEliahouFromentin2025JoA-verification} made use of known theoretical results on Wilf's conjecture to avoid the direct verifications of a significant number of numerical semigroups. It is worth noting that \(n_{100}\) is quite large and the maximum genus for which all the numerical semigroups have been explored is~\(76\). We observe that given a numerical semigroup \(S\) with Frobenius number \(f\), the maximum possible value of genus \(g(S)\) is \(f\) which corresponds to the case when \(S\) is \(\langle f+1,f+2, \ldots, 2f\rangle\). Therefore the verification of Wilf's conjecture up to genus less than or equal to 100 from~\cite{DelgadoEliahouFromentin2025JoA-verification} implies that the numerical semigroups with Frobenius number less than or equal to \(100\) satisfy Wilf's conjecture. 

Counting by maximum primitive was recently introduced in the second author's PhD thesis~\cite{Kumar2025phd-Numerical}. We denote the set of numerical semigroups with maximum primitive equal to \(n\) by \(\maxprimset{n}\), and its cardinality by \(\maxprimcard{n}\). The bounds for \(\maxprimcard{n}\) and its asymptotic behaviour have been discussed in~\cite{DelgadoKumarMarion2025pp-counting}. Moreover the monotonicity of the sequences \((\maxprimcard{2n})\) and \((\maxprimcard{2n+1})\) has been proven in~\cite{Kumar2025CiA-Monotonicity}. We are interested in exploring the properties of the set \(\maxprimset{n}\) for \(n\leq 60\), and in particular in computing its cardinality. Since most of the numerical semigroups in \(\maxprimset{n}\) for large \(n\) have primitive depth equal to 2 or 3 (see~\cite[Proposition 7.2]{DelgadoKumarMarion2025pp-counting}), finding a formula for the number of primitive depth 2 semigroups in \(\maxprimset{n}\) allows us to avoid a lot of computation. Let \(\frobset{f}(d)\) denote the set of numerical semigroups with Frobenius number \(f\) and depth \(d\), and \( \frobcard{f}(d)\) its cardinality. Let \(\maxprimset{n}(d)\) denote the set of numerical semigroups with maximum primitive \(n\) and primitive depth \(d\), and \(\maxprimcard{n}(d)\) its cardinality. We now prove the following generalization of \cite[Theorem 1.5]{DelgadoKumarMarion2025pp-counting} related to the cardinality of numerical semigroups with a given primitive depth. In particular we obtain a formula for \(\maxprimcard{n}(2)\). Let \(\mu\) denote the Möbius function.

\begin{proposition}\label{prop:cardinality_of_depth-pdepth_bijection} Given \(n>2\), and \(k\) positive, we have 
	\[A_n(k) = \sum_{d\mid n} \mu \left(\frac{n}{d}\right) \cdot N_{d}(k)\]
	and in particular \[A_n(2) = \sum_{d\mid n} \mu \left(\frac{n}{d}\right) \cdot \left(2^{\lfloor\frac{d-1}{2} \rfloor} -1\right).\]
\end{proposition} 

We count the numerical semigroups in \(\maxprimset{61}\) and \(\maxprimset{62}\) by including the above result, and  obtain the following result.

\begin{proposition}\label{prop:new_values_An_Nn}
	The value of \(\maxprimcard{61}\) is \(\num{2640706082}\) and \(\maxprimcard{62}\) is \(\num{2606696049}\). Moreover the value of \(\frobcard{61}\) is \(\num{2640706083}\) and \(\frobcard{62}\) is \(\num{2606766903}\). 
\end{proposition}
 
 We also explore \(\maxprimset{n}\) for possible counterexample of Wilf's conjecture. We first recall that all numerical semigroups with multiplicity \(m\) less than or equal to 19 satisfy Wilf's conjecture~\cite{KliemStump2022DCG-new}. We next discuss how this result allows us to conclude that Wilf's conjecture holds for numerical semigroups with certain small values of maximum primitives. Since the maximum primitive \(n\) is greater than or equal to \(m\), we observe that the numerical semigroups in \(\maxprimset{n}\) satisfy Wilf's conjecture for all \(n\) less than or equal to 19. We further consider the case of \(\maxprimset{n}\) for \(n\) in \(\{20, 21,22,23\} \) for which the Wilf's conjecture follows from known results. The multiplicity of any numerical semigroup in \(\maxprimset{20}\) is at most 19, so it satisfies Wilf's conjecture. For a numerical semigroup \(S\) in \(\maxprimset{21}\) or \(\maxprimset{22}\), either \(m(S)\) is less than 20 and thus \(S\) satisfies Wilf's conjecture~\cite{KliemStump2022DCG-new}, or the primitives of \(S\) are contained in the set \(\{20, 21, 22\} \) which implies that the embedding dimension of \(S\) is 2 or 3 and thus it satisfies Wilf's conjecture by~\cite{Sylvester1882AJM-Subvariants} or~\cite{FroebergGottliebHaeggkvist1987SF-numerical}. For \(S\) in \(\maxprimset{23}\), either we have \(m(S)\) less than 20 for which \(S\) satisfies Wilf's conjecture by~\cite{KliemStump2022DCG-new}, or \(m(S)\) is greater than or equal to \(20\) in which case we either have embedding dimension less than four or \(S\) is \(\langle 20,21,22,23 \rangle\). The former is known to satisfy Wilf's conjecture by ~\cite{KliemStump2022DCG-new}, and the primitives of \(\langle 20,21,22,23 \rangle\) form an arithmetic progression, so it satisfies Wilf's conjecture by~\cite{Sammartano2012SF-Numerical}. 

We next consider the numerical semigroups \(S\) in \(\maxprimset{24}\) which is the first non-trivial case that is not covered under known results. We note that if \(e(S)\) is less then or equal to 3 or if \(m(S)\) is less than 20 then Wilf's conjecture follows from~\cite{Sylvester1882AJM-Subvariants},~\cite{FroebergGottliebHaeggkvist1987SF-numerical} and~\cite{KliemStump2022DCG-new}. We therefore now consider the cases when \(e(S)\) is at least 4 and \(m(S)\) is at least 20. If \(e(S)\) is equal to 5 then \(S\) is a numerical semigroup generated by an arithmetic sequence which is known to satisfy Wilf's conjecture, so we further assume that \(e(S)\) is equal to 4. Moreover if \(m(S)\) is greater than 20 then we have \(\langle 21,22,23, 24 \rangle \) which has its primitives forming an arithmetic sequence. Thus we finally consider \(m(S)\) equal to 20 and \(e(S)\) equal to 4, and we have two possibilities: \(S\) is either \(\langle 20, 21,22, 24 \rangle \) or \(\langle 20,22,23, 24 \rangle \). We check these cases explicitly as a part of our work and observe that they indeed satisfy Wilf's conjecture. Thus all numerical semigroups in \(\maxprimset{24}\) satisfy Wilf's conjecture. We use computations to obtain the main result of the article.

\begin{theorem}\label{thm:verification_Wilf_upto_maxprim_60}
	Given a positive integer \(n\leq 60\), every numerical semigroup in \(\maxprimset{n}\) satisfies Wilf's conjecture. 
\end{theorem}

We show later that this finite family of numerical semigroups contains new examples that do not belong to the important classes of numerical semigroups that are known to satisfy Wilf's conjecture through the previously known results.

The outline of the article is as follows. 
In Section~\ref{sec:results} we provide some simple theoretical results.
In Section~\ref{sec:naive_algorithm} we present a naive algorithm to illustrate the key idea,
 and we devote Section~\ref{sec:optimized_algorithm} to the optimized algorithm. Section~\ref{sec:optimizations_for_counting} contains the special case of the optimized code for counting by maximum primitive. In addition we give theoretical calculations. 
Finally in Section~\ref{sec:computations} we discuss the results of the computations.

%% file: results.tex
We state here a few simple or well known results that will help us in proving the correctness of our algorithms.
For basics on numerical semigroups see the book by Rosales and García-Sánchez~\cite{RosalesGarcia2009Book-Numerical}.

\begin{lemma}\label{lemma:basics}
	Let \(S\) be a numerical semigroup with maximum primitive \(M\) and multiplicity \(m\).
	\begin{enumerate}
		\item No divisor of \(M\) smaller than \(M\) belongs to \(S\).
		\item No multiple of \(m\) bigger than \(m\) is a minimal generator of \(S\).
	\end{enumerate}
\end{lemma}

\begin{lemma}\label{lemma:depth2-is-easy}
	Let \(m\) and \(M\) be positive integers such that \(m>M/2\) and let \(Y\) be a subset of the range \((m,M)\). Then \(\langle Y\cup\{m,M\}\rangle\in \maxprimset{M,m}\) if and only if \(\gcd(Y\cup\{m,M\})=1\). 
\end{lemma}

\begin{lemma}\label{lemma:M-lcomb}
	Let \(S\in \maxprimset{M}\). If \(x\in S\setminus\{0,M\}\), then \(M-x\not\in S\).
\end{lemma}
\begin{proof}
	Suppose that \(M-x\in S\). But then \(M=(M-x)+x\), contradicting the hypothesis that \(M\) is primitive.
\end{proof}

\begin{lemma}\label{lemma:primitives_incongruent_mod_m}
	Let \(S\) be a numerical semigroup with multiplicity \(m\). The primitives of \(S\) are not congruent modulo \(m\). In particular, the embedding dimension of a numerical semigroup with multiplicity \(m\) does not exceed \(m\).
\end{lemma}

\begin{lemma}\label{lemma:coprime-generate-ns}
	A set \(Y\) of positive integers generates a numerical semigroup if and only if \(\gcd(Y)=1\).
\end{lemma}

%% file: naive_algorithm.tex
We start with a brute-force algorithm and then we move to a naive one by including a few simple observations.
An implementation and lots of patience led the authors to compute \(\maxprimcard{n}\) for \(n\le 60\) in~\cite{DelgadoKumarMarion2025pp-counting}. 

The pseudo-code we use for the description of the algorithms in the present paper has reminiscences of the \textsf{GAP} programming language. 
For the algorithms in the present paper we fix both the multiplicity \(m\) and the maximum primitive \(M\). 
Denote by \(\maxprimset{M,m}\) the set of numerical semigroups with multiplicity \(m\) and maximum primitive \(M\).
Notice that \(\maxprimset{M}=\bigcup_{1\le m\le M}\maxprimset{M,m}\), and the union is disjoint.

Algorithm~\ref{alg:brute-force} computes the set of numerical semigroups with multiplicity \(m\) and maximum primitive \(M\). 

\begin{algorithm}[h]  \caption{Brute-force algorithm to compute \(\maxprimset{M,m}\)}\label{alg:brute-force}
	\DontPrintSemicolon
	\SetKwComment{cgap}{\# }{} 
	\KwData{Positive integers \(M,m\), with \(m<M\)}\cgap{m stands for the multiplicity and M for the maximum primitive}
	\KwResult{\(\maxprimset{M,m}\)}
	
	\(set := [\quad]\)\;
	\nl\For{Y in \(\mathrm{Combinations}([m+1..M-1])\)}{\label{alg:brute-force:loop}
		\(P:=Y\cup \{m,M\}\)\;
		\If{\(\mathrm{Gcd}(P)=1\)}{\cgap{test if \(P\) generates a numerical semigroup}
			\If{\(P=\mathrm{MinimalGenerators}(\langle P \rangle)\)}{
				\(\mathrm{Add}(set,P)\)\;
			}
		}
	}
	\KwRet \(set\)
\end{algorithm}
Algorithm~\ref{alg:brute-force} computes \(\maxprimset{M,m}\) by looping (Line~\ref{alg:brute-force:loop}) through all the \(2^{M-m-1}\) subsets of \(\{m+1,\ldots,M-1\}\) and making some checks. This is infeasible for (not too) large \(M-m\).

Algorithm~\ref{alg:naive} is an optimization of Algorithm~\ref{alg:brute-force} based on simple observations. 

\begin{algorithm}[h]  \caption{A naive algorithm to compute \(\maxprimset{M,m}\)}\label{alg:naive}
	\DontPrintSemicolon
	\SetKwComment{cgap}{\# }{} 
	\KwData{Positive integers \(M,m\), with \(m<M\)}
	\KwResult{\(\maxprimset{M,m}\)}
	\(set := [\quad]\)\;
	\nl \If{m>M/2}{\hfill\cgap{the primitive depth 2 case}\label{alg:naive:depth2}
	\nl	\For{Y in \textrm{Combinations}([m+1,M-1])}{\label{alg:naive:depth2-loop}
			\(P:=Y\cup \{m,M\}\)\;
			\If{\(\mathrm{Gcd}(P)=1\)}{
				\(\mathrm{Add}(set,P)\)
			}
		}
	\KwRet \(set\)
	}
	\nl \(div := \mathrm{Divisors}(M)\) \hfill \cgap{divisors of M} \label{alg:naive:divisors}
	\nl \(mult := \{km : 1 < k < M/m\}\) \hfill \cgap{multiples of m in (m,M)}\label{alg:naive:multiplesm}
	\nl \(X := \{m,\ldots,M\}\setminus(div\cup mult)\)\;\label{alg:naive:exclude_divs_multiples}
	\nl \For{\(i\in\{1,\ldots,\min(m-2,|X|)\}\)}{\label{alg:naive:emb_dim}
	\nl	\(comb := \mathrm{Combinations}(X,i)\)\;\label{alg:naive:combinations}
	\nl	\For{Y in comb}{\label{alg:naive:loop}
		\(P:=Y\cup \{m,M\}\)\;
	\nl	\If{\(\mathrm{Gcd}(P)=1\)}{\label{alg:naive:isnumericalsgp}
	\nl		\If{\(P=\mathrm{MinimalGenerators}(\langle P \rangle)\)}{\label{alg:naive:Pminimal}
	\nl			\(\mathrm{Add}(set,P)\)\;\label{alg:naive:add_semigroup}
		
				}
			}
		}
	}
	\KwRet \(set\)
\end{algorithm}

Some comments on Algorithm~\ref{alg:naive} follow:
\begin{itemize}
	\item The \emph{for-block} in the loop in Line~\ref{alg:naive:depth2-loop} basically reduces to computing a greatest common divisor, by Lemma~\ref{lemma:depth2-is-easy}.
	\item Multiples of \(m\) and divisors of \(M\) are not primitives of any semigroup in \(\maxprimset{M,m}\), by Lemma~\ref{lemma:basics}. They are excluded in Line~\ref{alg:naive:exclude_divs_multiples}.
	\item In Line~\ref{alg:naive:emb_dim} it is used the fact that the embedding dimension of a numerical semigroup does not exceed its multiplicity, as stated in Lemma~\ref{lemma:primitives_incongruent_mod_m}.
	\item Line~\ref{alg:naive:combinations} computes the subsets of \(X\) with \(i\) elements.
	\item Line~\ref{alg:naive:isnumericalsgp} tests if \(Y\cup\{m,M\}\) generates a numerical semigroup.
	\item Line~\ref{alg:naive:Pminimal} tests whether \(Y\cup \{m,M\}\) is the (unique) minimal generating set of the numerical semigroup \(\langle Y\cup \{m,M\}\rangle\)
	\item Line~\ref{alg:naive:add_semigroup} adds the semigroup (given by its minimal generating set).  
\end{itemize}
From the comments made we easily deduce the following:
\begin{proposition}\label{prop:correction_naive}
	Algorithm~\ref{alg:naive} correctly produces the claimed output. 
\end{proposition}

In terms of practical efficiency, we start observing that the time required to determine the minimal generating set of a numerical semigroup is not negligible in general. 
As computing a greatest common divisor is inexpensive, the primitive depth 2 case is much faster in Algorithm~\ref{alg:naive}.
The \emph{for-block} in the loops in Lines~\ref{alg:brute-force}.\ref{alg:brute-force:loop} and~\ref{alg:naive}.\ref{alg:naive:loop} are equal. But, in the case of Algorithm~\ref{alg:naive} the looping is over a much smaller set in general (note that the primitive depth~2 case is treated separately). 
The loop in Line~\ref{alg:naive}.\ref{alg:naive:loop} is executed several times (at most \(m-2\)).

Algorithm~\ref{alg:naive} was implemented using the \texttt{GAP}~\cite{GAP4.15.1} language and is part of the \texttt{GAP} package \texttt{numericalsgps}~\cite{NumericalSgps1.4.0}.

%% file: optimized_algorithm.tex
As already referred, the inefficiency of the algorithms in Section~\ref{sec:naive_algorithm} is caused in part by the fact that large sets have too many subsets. The purpose of Algorithm~\ref{alg:possible_large_primitives} is to reduce the size of some sets we have to deal with (in the same spirit of the naive optimization already considered in Section~\ref{sec:naive_algorithm}).

The data available in Table~\ref{table:counting_by_maxprim-and-frob}, in Section~\ref{sec:computations}, was obtained using the algorithms presented here.

Let \(P=\{m=p_1,\ldots,p_n=M\}\), with \(n>1\), be a set of integers such that \(p_1<\cdots<p_n\). Algorithm~\ref{alg:possible_large_primitives} computes a subset of the open interval \(\left(p_{n-1}, p_n\right)\) obtained by removing elements larger than \(p_{n-1}\) that are ``obviously'' not primitives of a numerical semigroup containing the elements of \(P\) as primitives.
 
\begin{algorithm}[h] \caption{PossibleLargePrimitives}\label{alg:possible_large_primitives}
	\DontPrintSemicolon
	\SetKwComment{cgap}{\# }{} 
	\KwData{A set \(P=\{m=p_1,\ldots,p_n=M\}\) of positive integers, with \(p_1<\cdots<p_n\) and \(n>1\).\hfill \cgap{m stands for multiplicity and M for maximum primitive}}
	\KwResult{A subset of \(\left(p_{n-1}, p_n\right)\) without some obviously non-primitives of a numerical semigroup containing \(P\) in its set of primitives}

	\nl	\(lcombs = \langle P \rangle\cap [0..M]\) \hfill \cgap{linear combinations of elements of \(P\)}\label{alg:possible-large-primitives:lcombs}
	\nl	\(PP = \left(p_{n-1}, p_n\right)\setminus (lcomb\cup M-lcomb)\)\; \label{alg:possible-large-primitives:linear_combinations}
	\(Imp = [\quad ]\)\hfill \cgap{list of some other non-primitives}
	\nl \For{p in PP}{			\label{alg:possible-large-primitives:further_optimization}
			\If{\(M \in \langle Q\cup {p}\rangle\)}{
				Add(\(Imp,p\))
			} 
		}
	\KwRet \(PP\setminus Imp\)
\end{algorithm}

Comments on Algorithm~\ref{alg:possible_large_primitives}
\begin{itemize}
	\item Line~\ref{alg:possible-large-primitives:linear_combinations} ... Note the use of Lemma~\ref{lemma:M-lcomb}
	\item Line~\ref{alg:possible-large-primitives:further_optimization} ... This loop is used to find some possible elements not yet removed that can not be primitives
\end{itemize}
From the comments made we easily deduce the following:
\begin{proposition}\label{prop:possible_large_primitives}
	Algorithm~\ref{alg:possible_large_primitives} correctly eliminates from \(\left(p_{n-1}, p_n\right)\) some elements that can not primitives in any semigroup whose set of primitives contains the input. 
\end{proposition}

We call \texttt{PossibleLargePrimitives} to a function implementing Algorithm~\ref{alg:possible_large_primitives}. It will be used in the sequel.

\medskip

Let \(P=\{m=p_1,\ldots,p_n=M\}\), with \(n>1\), be a set of integers such that \(p_1<\cdots<p_n\). Algorithm~\ref{alg:ns_containing_given_primitives} computes the numerical semigroups whose set of primitives Q satisfies: \(P\subseteq Q\subseteq P\cup[p_{n-1},M]\).

\begin{algorithm}[h]  \caption{Algorithm to compute numerical semigroups containing given primitives and other ``large'' primitives.}\label{alg:ns_containing_given_primitives}
	\DontPrintSemicolon
	\SetKwComment{cgap}{\# }{} 
	\KwData{A set \(P=\{m=p_1,\ldots,p_n=M\}\) of positive integers, with \(p_1<\cdots<p_n\) and \(n>1\).}\cgap{m stands for multiplicity and M for maximum primitive}
	\KwResult{Computes the numerical semigroups whose set of primitives \(Q\) satisfies: \(P\subseteq Q\subseteq P\cup[p_{n-1},M]\).}\;
	A := PossibleLargePrimitives(P) \hfill \cgap{It is used Algorithm~\ref{alg:possible_large_primitives}.}
	listsprim := [\quad ] \hfill\cgap{to store the numerical semigroups computed.}
	\nl partition := a partition of A into residue classes modulo \(m\) \label{alg:ns_containing_given_primitives:partition}\;
	\nl \For{k in [1..m-n]}{ \cgap{note that a numerical semigroup has at most m primitives}\label{alg:ns_containing_given_primitives:edim}
	\nl		iter := Combination(partition,k) \cgap{compute the subsets of size k of partition}  \label{alg:ns_containing_given_primitives:combination}
			\For{it in iter}{
	\nl			cartesian := CartesianProduct(it) \label{alg:ns_containing_given_primitives:cartesian}
				\For{C in cartesian}{
					\If{\(Gcd(P\cup C) = 1\)}{
						mingens := MinimalGenerators(\(\langle P\cup C \rangle\))\;
	\nl						\If{\(P\cup C\) = mingens}{\label{alg:ns_containing_given_primitives:inserttest}
							\cgap{a test can be inserted here} 
	\nl						Add(listsprim,mingens)\label{alg:ns_containing_given_primitives:addset}
						}
					} 
				}
			}
		} 
	\KwRet listsprim
\end{algorithm}

Comments on Algorithm~\ref{alg:ns_containing_given_primitives}
\begin{itemize}
	\item Line~\ref{alg:ns_containing_given_primitives:partition} is used for taking advantage of the fact that a residue class modulo the multiplicity has at most one primitive (see Lemma~\ref{lemma:primitives_incongruent_mod_m}).
	\item Line~\ref{alg:ns_containing_given_primitives:edim}: note that the semigroups to be computed have at most \(m\) primitives, n of which are given, and use Lemma~\ref{lemma:primitives_incongruent_mod_m}.
	\item Line~\ref{alg:ns_containing_given_primitives:cartesian} To take exactly one element in each residue class.
    \item Line~\ref{alg:ns_containing_given_primitives:inserttest} Testing some property (for instance Wilf's inequality) can be inserted in this \texttt{if} clause.
	\item Line~\ref{alg:ns_containing_given_primitives:addset} Adds mingens to listsprim.
\end{itemize}

From the comments made we easily deduce the following:
\begin{proposition}\label{prop:ns_containing_given_primitives}
	Algorithm~\ref{alg:ns_containing_given_primitives} correctly computes the claimed output.
\end{proposition}

We call \texttt{NSgpsWithGivenPrimitives} to a function implementing Algorithm~\ref{alg:ns_containing_given_primitives}.

\medskip

Algorithm~\ref{alg:tree_rooted_nM} is used to construct a tree of submonoids of \(\mathbb{N}\) rooted by some given submonoid. 
It uses the recursive local function \texttt{children}.
There are two kinds of leaves.

The nodes are submonoids of \(\mathbb{N}\), some of which are numerical semigroups.

The leaves are monoids with less than \texttt{len} possible large primitives. Observe that for \(\text{len}=1\) the whole tree is constructed. Experiments suggest that the \texttt{if} clause that appears in Line~\ref{alg:tree_rooted_nM}.\ref{alg:ns_containing_given_primitives:len_if_clause} may constitute an optimization. For values of maximum primitive around \(60\) we got better results for \(\text{len}\sim 14(M-m)/5\).

\begin{algorithm}[h]  \caption{Algorithm to explore a tree of submonoids of \(\mathbb{N}\)}\label{alg:tree_rooted_nM}
	\DontPrintSemicolon
	\SetKwComment{cgap}{\# }{} 
	\KwData{Positive integers \(m\) and \(M\), with \(m<M\). A record opt of options. The field opt.len is a positive integer}\cgap{m stands for multiplicity and M for maximum primitive}
	\KwResult{\(\maxprimset{M,m}\)}\;
	
	\SetKwFunction{proc}{}
	\SetKwProg{myproc}{children(X)}{}{}\cgap{A local recursive function}
	\nl \myproc{\proc}{ \label{alg:tree_rooted_nM:children}
		\nl \(XX := \mathrm{PossibleLargePrimitives}(X)\)\hfill \cgap{Use Algorithm~\ref{alg:possible_large_primitives}}\;
		\For{r in XX}{
			\(gens := X\cup \{r\}\)\;
			\If{\(\gcd(gens)=1\)}{AddSet(LIST,\(gens\))}\;
	\nl		\eIf{Length(\(\mathrm{PossibleLargePrimitives}\)(gens)) > len}{ \label{alg:ns_containing_given_primitives:len_if_clause}
	\nl			children(gens) \cgap{continue the construction of the tree}
			}{
			\(listsprim := \mathrm{NSgpsWithGivenPrimitives}(gens,opt)\) \hfill\cgap{Use Algorithm~\ref{alg:ns_containing_given_primitives}}
			\(LIST = LIST \cup listsprim\)
			}
		}
	}\;

	LIST = [\quad ] \hfill \cgap{stores the numerical semigroups computed.}
\nl	\eIf{\(m>M/2\)}{ \label{alg:ns_containing_given_primitives:depth2}\hfill\cgap{the maxprim depth 2 case, as in Algorithm~\ref{alg:naive}}
		\For{Y in Combinations([m+1,M-1])}{
			prims = \([m..M]\cup Y\)\;
			\If{\(\mathrm{Gcd}(prims)=1\)}{\label{alg:ns_containing_given_primitives:insert_test}
	\nl						\cgap{tests can be inserted here}				
				\(\mathrm{Add}(\mathrm{LIST},prims)\)				
			}
		}
	}
	{
		\If{Gcd(m,M)=1}{
			Add(LIST,[m,M])
		}
		children([m,M])\;	
}
	\KwRet LIST
\end{algorithm}

	Comments on Algorithm~\ref{alg:tree_rooted_nM}
\begin{itemize}
	\item Line~\ref{alg:tree_rooted_nM:children} The local recursive function affects the global variable LIST.
	\item Line~\ref{alg:ns_containing_given_primitives:insert_test} When only counting is in cause, the \texttt{if} block can be drastically simplified by using the result in Section~\ref{sec:optimizations_for_counting}.
\end{itemize}

From the comments made we easily deduce the following:
\begin{proposition}\label{prop:tree_rooted_nM}
	Algorithm~\ref{alg:tree_rooted_nM} correctly computes the set \(\maxprimset{M,m}\) of numerical semigroups with multiplicity \(m\) and maximum primitive \(M\).
\end{proposition}

\medskip

The use of the tree in (Algorithm~\ref{alg:tree_rooted_nM}) permits to use a small amount of memory. In the same direction, one can take advantage of the use of iterators (e.g. the \textsf{GAP} iterators \texttt{IteratorOfCombinations} and \texttt{IteratorOfCartesianProduct}).

We implemented our algorithms in the GAP language, which is high level one. We make use of functions available in GAP and in the numericalsgps GAP package. This has serious performance limitations. The numbers achieved are nevertheless interesting. 
Not negligible is also the fact that the code, being written in a high level programming language, is easy to verify. The numbers obtained can be used by any one whishing to do faster implementations to ensure the absence of bugs.   

%% file: optimizations_for_counting.tex
This section is devoted to optimizations that can be made when having just counting purposes. In particular, we prove Proposition \ref{prop:cardinality_of_depth-pdepth_bijection} and use it to avoid the need to explore the primitive depth 2 numerical semigroups while counting.

Given a rational number \(r\) and a set \(A\subset \mathbb{N}\), we denote the set \( \{ra: a\in A\} \) by \(r\cdot A\) or \(\frac{A}{1/r}\). Given a numerical semigroup \(S\), let \(\leftsoper(S)\) denote the set of left elements \(\{ s\in S: s< \Frobeniusoper(S)\}\), and \(\leftsoper'(S)\) denote \(\leftsoper(S) \cup \{\Frobeniusoper(S)\} \). Let \(\mu\) be the Möbius function. We define the map \(\Phi: \maxprimset{n} \to  \frobset{n} \) by \( \Phi(S)= (S\setminus \{n\} ) \cup (n,\infty) \). We prove the following result.

\begin{proposition}\label{prop:size_of_prim_depth_2}
	Given a positive integer \(n>2\), the number of subsets of \((n/2,n]\) that contain \(n\) and have the greatest common divisor equal to \(1\) is equal to 
	\[ \sum_{d\mid n} \mu \left(\frac{n}{d}\right) \cdot \left(2^{\lfloor\frac{d-1}{2} \rfloor} -1\right).\]
\end{proposition}
We first prove the following lemma to prove the above proposition.    
\begin{lemma}\label{lemma:bijection_of_depth_and_prim_depth} Let \(n>2\) and \(k\leq n\). Then we have a bijection   
	\[ \Psi: \frobset{n}(k) \to \bigcup_{d\mid n} \maxprimset{d}(k) \]
	defined by
	\[ \Psi(S) = \left\langle \frac{\leftsoper'(S) }{\gcd( \leftsoper'(S))} \right\rangle. \]
\end{lemma}
\begin{proof} 
	We first consider the injectivity of \(\Psi\). Let \(S,R \in \frobset{n} (k) \) such that \(\Psi(S)=\Psi(R)\). Thus 
	\[ \left\langle \frac{\leftsoper'(S) }{\gcd( \leftsoper'(S))} \right\rangle = \left\langle \frac{\leftsoper'(R) }{\gcd( \leftsoper'(R))} \right\rangle.\]
	We note that \(n\) is the largest element in both \(\leftsoper'(S)\) and \(\leftsoper'(R)\), and in either of the sets \(n\) cannot be written as a sum of two non-zero elements. Thus the maximum primitive of both \(\Psi(S)\) and \(\Psi(R)\) is equal to \(n/\gcd( \leftsoper'(S))\) which is also equal to \(n/\gcd( \leftsoper'(R))\). Thus \(\gcd( \leftsoper'(S))\) is equal to \(\gcd( \leftsoper'(R))\), which we denote by \(x\). We therefore have
	\[ \left\langle \frac{\leftsoper'(S) }{x} \right\rangle = \left\langle \frac{\leftsoper'(R) }{x} \right\rangle\]
	and thus \( \left\langle \leftsoper'(S) \right\rangle = \left\langle \leftsoper'(R) \right\rangle \). Moreover since \(n\) is a primitive in both \(\left\langle \leftsoper'(S) \right\rangle\) and \(\left\langle \leftsoper'(R) \right\rangle\) we have \(\left\langle \leftsoper'(S)\setminus \{n\} \right\rangle = \left\langle \leftsoper'(R) \setminus \{n\} \right\rangle \) or \(\left\langle \leftsoper(S)\right\rangle = \left\langle \leftsoper(R) \right\rangle \). Thus
	\[ S = \left\langle \leftsoper(S) \cup (n, \infty)\right\rangle = \left\langle \leftsoper(R) \cup (n, \infty) \right\rangle = R \]
	which shows the injectivity of \(\Psi\).
	
	We next look at the surjectivity of \( \Psi \). Given \(T\in \maxprimset{d}(k)\), let \(T'\) be the set \( \left(\langle \frac{n}{d} \cdot \primitivesoper(T) \rangle\setminus \{n\}\right)  \cup (n, \infty)\). We observe that \(T'\) contains \(0\), and it is cofinite. Moreover, the sets \(\left(\langle \frac{n}{d} \cdot \primitivesoper(T) \rangle\setminus \{n\}\right)\) and \((n, \infty)\) are both additively closed, and the sum of an element of \(\left(\langle \frac{n}{d} \cdot \primitivesoper(T) \rangle\setminus \{n\}\right)\) with an element of \((n,\infty)\) belongs in \((n, \infty)\). Thus \(T'\) is an additively closed set, and hence it is a numerical semigroup. We observe that \( \leftsoper(T')\) is \(\langle \frac{n}{d} \cdot \primitivesoper(T) \rangle \cap [0,n) \) and \(\Frobeniusoper(T')\) is \(n\). Thus
	\[ \Psi ( T' ) = \left\langle \frac{ ((\langle \frac{n}{d} \cdot \primitivesoper(T) \rangle) \cap [0,n)) \cup \{n\} }{n/d} \right\rangle = \left\langle \frac{ \frac{n}{d} \cdot \primitivesoper(T) }{n/d} \right\rangle = \langle \primitivesoper (T) \rangle = T. \]
	This shows the surjectivity, and thus \( \Psi \) is a bijection.
\end{proof}

We now prove Proposition \ref{prop:cardinality_of_depth-pdepth_bijection}.\\

\noindent\emph{Proof of Proposition \ref{prop:cardinality_of_depth-pdepth_bijection}:}
	We consider the cardinality of the sets in the bijection \(\Psi\) in Lemma \ref{lemma:bijection_of_depth_and_prim_depth} and conclude that 
	\[ \frobcard{n}(k) = \sum_{d\mid n} \maxprimcard{d}(k). \]
	We apply the Möbius inversion theorem~\cite[Theorem 266]{HardyWright2008Book-introduction} to this, and obtain
	\[A_n(k) = \sum_{d\mid n} \mu \left(\frac{n}{d}\right) \cdot N_{d}(k).\]
	We next consider the case when \(k\) is equal to 2. We observe that the depth 2 numerical semigroups with Frobenius number \(d\) are of the form \( \{0\} \cup X \cup (d,\infty) \), where \(X\) is a non empty subset of \( (d/2, d) \). Thus \( N_{d}(2) \) is equal to \( 2^{\lfloor\frac{d-1}{2} \rfloor} -1 \). Thus
	\[A_n(2) = \sum_{d\mid n} \mu \left(\frac{n}{d}\right) \cdot \left(2^{\lfloor\frac{d-1}{2} \rfloor} -1\right). \qedhere \]

\noindent \textit{Proof of Proposition \ref{prop:size_of_prim_depth_2}}:
We observe that every subset \(A\) of \((n/2,n]\) that contains \(n\) and has its greatest common divisor equal to \(1\) gives us a unique numerical semigroup \(\langle A\rangle\) with maximum primitive \(n\) and primitive depth $2$. Thus the number of such subsets is precisely \(A_n(2)\), and the result follows from Proposition \ref{prop:cardinality_of_depth-pdepth_bijection}.      \qed

%% file: computations.tex
In this section we discuss the results of our computations and, in particular, we provide the Table \ref{table:counting_by_maxprim-and-frob} with the values \(\maxprimcard{61}\) and \(\maxprimcard{62}\). We also discuss some patterns in the data. 
The following result enlists the known major classes of numerical semigroups that satisfy Wilf's conjecture.
\begin{theorem}\label{thm:Wilf_known_cases}
	Wilf's conjecture holds for all numerical semigroups satisfying one of the following conditions:
	\begin{enumerate}
		\item \(e \le 3\)~\cite{FroebergGottliebHaeggkvist1987SF-numerical}
		\item \(c \le 3m\)~\cite{Eliahou2018JEMS-Wilfs} 
		\item \(e \ge m/3\)~\cite{Eliahou2020EJC-graph}
		\item \(|L| \le 12\)~\cite{EliahouMarin-Aragon2021CiA-numerical}
		\item \(m \le 19\)~\cite{KliemStump2022DCG-new}
		\item \(|(m,2m)|\geq \sqrt{3m} \)~\cite{DelgadoKumarMarion2025pp-counting}
		\item \(g \le 100\)~\cite{DelgadoEliahouFromentin2025JoA-verification}
		\item Elements of \(P\) are in an arithmetic progression.~\cite{Sammartano2012SF-Numerical}
	\end{enumerate}
\end{theorem}

 We now show for the sake of non-triviality of our result that, in particular, these classes do not contain \(\maxprimset{60}\).  We consider the numerical semigroup \(\langle 50,52,53, 60 \rangle \) and denote it by \(S\). Since \(S\subset I=\langle[50,60]\rangle\), we observe that \(\Frobeniusoper(S)\geq \Frobeniusoper(I)=249 \) and \(g(S)\geq g(I)>100\), whereas \(\multiplicityoper(S)=\multiplicityoper(I)=50\). Also we have \(c(S)\geq c(I)=250\geq 3\multiplicityoper(S)= 150\). We note that the embedding dimension \(e(S)\) is 4, which is greater than 3 but less than \(m(S)/3\). Moreover, the number of primitives of \(S\) in the interval \((m(S), 2m(S))\) is less than \(\sqrt{3n}= \sqrt{180}\), where \(n\) denotes the maximum primitive of \(S\). Thus \(S\) does not satisfy any of the condition of Theorem \ref{thm:Wilf_known_cases}. We finally observe that many such numerical semigroups may be constructed with the same property by choosing sets of primitives, say \(A\), from the integer interval \([50,60]\) such that 50 and 60 belong to \(A\) with at least two more elements, the elements of \(A\) are not in an arithmetic progression, and such that the greatest common divisor of \(A\) is 1.\\
 
 We first observe that the growth in the values of \(\maxprimcard{n}\) is exponential just like the case of \(\frobcard{n}\), as is known from \cite[Corollary 1.4]{DelgadoKumarMarion2025pp-counting}. We note the linear pattern 
 in Figure~\ref{fig:log_scale}.
 
 We now discuss some patterns in the ratio of consecutive terms of the sequence \( (\maxprimcard{n})\). As the ratio \(\frobcard{n}/\maxprimcard{n}\) tends to 1 when \(n\) tends to infinity (see \cite[Corollary 1.4]{DelgadoKumarMarion2025pp-counting}), the asymptotic limit of \(\frobcard{n+1}/\frobcard{n} \) is same as that of \(\maxprimcard{n+1}/\maxprimcard{n} \). We observe from Figure~\ref{fig:ratios} that the ratio \(\maxprimcard{2n+1}/\maxprimcard{2n} \) tends to~2 whereas the ratio \(\maxprimcard{2n}/\maxprimcard{2n+1} \) tends to~1 as conjectured in equivalent terms in \cite[Conjecture 20]{Kumar2025CiA-Monotonicity}.  
 
 The numbers \(\maxprimcard{n}\) in Table~\ref{table:counting_by_maxprim-and-frob} were computed using our algorithms.
 The values of \(\maxprimcard{61}\) and \(\maxprimcard{62}\) are new. This allows us to prove Proposition~\ref{prop:new_values_An_Nn}.\\
 
 \noindent\textit{Proof of Proposition \ref{prop:new_values_An_Nn}: } We obtain the values of \(\maxprimcard{61}\) and \(\maxprimcard{62}\) by computation. Moreover by \cite[Theorem 1.5]{DelgadoKumarMarion2025pp-counting}, we have for \(n>0\)
 \[ \frobcard{n} = \sum_{d\mid n} \maxprimcard{\frac{n}{d}} \]
 and we use it together with the values from Table \ref{table:counting_by_maxprim-and-frob} to compute the values of
 \(\frobcard{61}\) and \(\frobcard{62}\). \qed\\
 
 \begin{table}[h]
 	\centering
 	\setlength{\tabcolsep}{8pt}
 	\setlength{\arrayrulewidth}{0.4mm}
 	\begin{tabular}{| r  r  r  | r  r  r  | r  r  r  |}
 		\hline
 		\(n\)&\( \maxprimcard{n} \)&\( \frobcard{n}  \) &\(n\)&\( \maxprimcard{n} \)&\( \frobcard{n}  \)  &\(n\) &\( \maxprimcard{n} \) &\( \frobcard{n}  \)  \\ \hline		
 		\rowcolor{black!5}
 		1 	 &1 	 &1 	 &26 	 &\num{8069} 	 &\num{8175} 	 &51 	 &\num{78793811} 	 &\num{78794277}\\
 		2 	 &0 	 &1 	 &27 	 &\num{16111} 	 &\num{16132} 	 &52 	 &\num{78922130} 	 &\num{78930306}\\
 		\rowcolor{black!5}
 		3 	 &1 	 &2 	 &28 	 &\num{16163} 	 &\num{16267} 	 &53 	 &\num{162306074} 	 &\num{162306075}\\
 		4 	 &1 	 &2 	 &29 	 &\num{34902} 	 &\num{34903} 	 &54 	 &\num{155991666} 	 &\num{156008182}\\
 		\rowcolor{black!5}
 		5 	 &4 	 &5 	 &30 	 &\num{31603} 	 &\num{31822} 	 &55 	 &\num{325800242} 	 &\num{325800297}\\
 		6 	 &2 	 &4 	 &31 	 &\num{70853} 	 &\num{70854} 	 &56 	 &\num{320507004} 	 &\num{320523279}\\
 		\rowcolor{black!5}
 		7 	 &10 	 &11 	 &32 	 &\num{68476} 	 &\num{68681} 	 &57 	 &\num{643198150} 	 &\num{643199112}\\
 		8 	 &8 	 &10 	 &33 	 &\num{137339} 	 &\num{137391} 	 &58 	 &\num{644611930} 	 &\num{644646833}\\
 		\rowcolor{black!5}
 		9 	 &19 	 &21 	 &34 	 &\num{140196} 	 &\num{140661} 	 &59 	 &\num{1317118755} 	 &\num{1317118756}\\
 		10 	 &17 	 &22 	 &35 	 &\num{292066} 	 &\num{292081} 	 &60 	 &\num{1269732856} 	 &\num{1269765591}\\
 		\rowcolor{black!5}
 		11 	 &50 	 &51 	 &36 	 &\num{269817} 	 &\num{270258} 	 &\textbf{61} 	 &\num{2640706082} 	 &\num{2640706083}\\
 		12 	 &35 	 &40 	 &37 	 &\num{591442} 	 &\num{591443} 	 &\textbf{62} 	 &\num{2606696049} 	 &\num{2606766903}\\
 		\rowcolor{black!5}
 		13 	 &105 	 &106 	 &38 	 &\num{581492} 	 &\num{582453} 	 &	 & 	 &\\
 		14 	 &92 	 &103 	 &39 	 &\num{1155905} 	 &\num{1156012} 	 &	 &	 &\\
 		\rowcolor{black!5}
 		15 	 &194 	 &200 	 &40 	 &\num{1160411} 	 &\num{1161319} 	 &	 &	 &\\
 		16 	 &195 	 &205 	 &41 	 &\num{2425710} 	 &\num{2425711} 	 &	 &	 &\\
 		\rowcolor{black!5}
 		17 	 &464 	 &465 	 &42 	 &\num{2285281} 	 &\num{2287203} 	 &	 &	 &\\
 		18 	 &382 	 &405 	 &43 	 &\num{4889433} 	 &\num{4889434} 	 &	 &	 &\\
 		\rowcolor{black!5}
 		19 	 &960 	 &961 	 &44 	 &\num{4783757} 	 &\num{4785671} 	 &	 &	 &\\
 		20 	 &877 	 &900 	 &45 	 &\num{9574948} 	 &\num{9575167} 	 &	 &	 &\\
 		\rowcolor{black!5}
 		21 	 &\num{1816} 	 &\num{1828} 	 &46 	 &\num{9674748} 	 &\num{9678844} 	 &	 &	 &\\
 		22 	 &\num{1862} 	 &\num{1913} 	 &47 	 &\num{19919901} 	 &\num{19919902} 	 &	 &	 &\\
 		\rowcolor{black!5}
 		23 	 &\num{4095} 	 &\num{4096} 	 &48 	 &\num{18893119} 	 &\num{18896892} 	 &	 &	 &\\
 		24 	 &\num{3530} 	 &\num{3578} 	 &49 	 &\num{40010840} 	 &\num{40010851} 	 &	 &	 &\\
 		\rowcolor{black!5}
 		25 	 &\num{8268} 	 &\num{8273} 	 &50 	 &\num{39437596} 	 &\num{39445886} 	 &   &   &\\
 		\hline	                  
 	\end{tabular}
 	
 	\caption{ Counting by maximum primitive and by Frobenius number\label{table:counting_by_maxprim-and-frob}}
 \end{table}

 Next we have a few images produced using the data in Table~\ref{table:counting_by_maxprim-and-frob}.
  
 \begin{figure}[h]
 	\centering
 \includegraphics[width=0.75\textwidth]{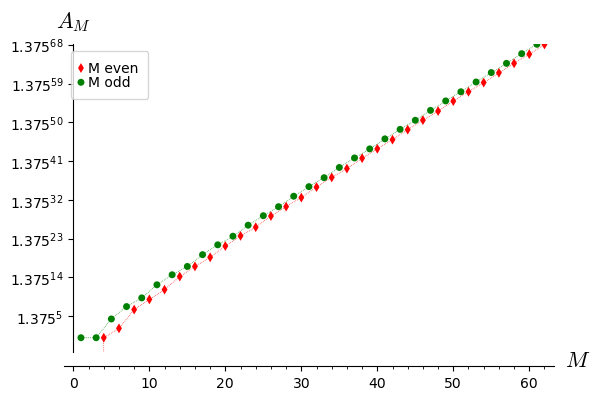}
 \caption{\label{fig:log_scale} \(\maxprimcard{M}\) using logarithmic scale in the \(y\)-axis} 
\end{figure} 
 
 \begin{figure}[h]
	\centering
 \includegraphics[width=0.75\textwidth]{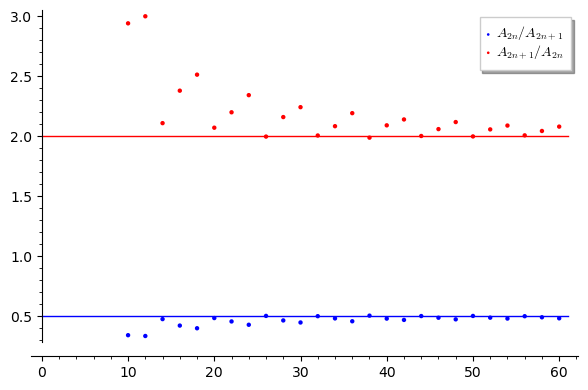}
  \caption{\label{fig:ratios} The behavior of the sequences of ratios \(\maxprimcard{2n}/\maxprimcard{2n+1}\) and \(\maxprimcard{2n+1}/\maxprimcard{2n}\)} 
\end{figure}